\RequirePackage{ifpdf} 
\ifpdf
\documentclass[11pt,pdftex]{amsart}
\else
\documentclass[11pt,dvips]{amsart}
\fi

\ifpdf
\fi
\usepackage[bookmarks, 
colorlinks=false, backref,plainpages=false]{hyperref}
\usepackage[english]{babel}

\usepackage{hyperref}
\usepackage[framed,numbered,autolinebreaks]{mcode}

\usepackage{graphics}
\usepackage{epsfig}
\usepackage{amsmath,amssymb,amsthm}

\renewcommand{\theequation}{\thesection.\arabic{equation}}

\newtheorem{thm}{Theorem}[section]

\newtheorem{lem}[thm]{Lemma}

\newtheorem{rem}[thm]{Remark}

\begin{document}
\newcommand{\BX}{{\bf X}}
\newcommand{\cv}{{\cal V}}
\newcommand{\cW}{{\cal W}}
\newcommand{\co}{{\cal O}}

\renewcommand{\theequation}{\thesection.\arabic{equation}}
\def\@eqnnum{{\reset@font\rm (\theequation)}}

\def\abstract{
\advance \rightskip by 10mm
\advance \leftskip by 10mm
\vspace{-0.8em}
\noindent
\small{\bf Abstract.}
}
\def\endabstract{\par\normalsize\rm}

\def\Xint#1{\mathchoice
{\XXint\displaystyle\textstyle{#1}}%
{\XXint\textstyle\scriptstyle{#1}}%
{\XXint\scriptstyle\scriptscriptstyle{#1}}%
{\XXint\scriptscriptstyle\scriptscriptstyle{#1}}%
\!\int}
\def\XXint#1#2#3{{\setbox0=\hbox{$#1{#2#3}{\int}$}
\vcenter{\hbox{$#2#3$}}\kern-.5\wd0}}
\def\ddashint{\Xint=}
\def\dashint{\Xint-}

\def\a{\alpha}
\def\b{\beta}
\def\d{\delta}\def\D{\Delta}
\def\e{\epsilon}
\def\g{\gamma}\def\G{\Gamma}
\def\k{\kappa}
\def\lam{\lambda}\def\Lam{\Lambda}
\renewcommand\o{\omega}\renewcommand\O{\Omega}
\def\s{\sigma}\def\S{\Sigma}
\renewcommand\t{\theta}\def\vt{\vartheta}
\newcommand{\vphi}{\varphi}
\def\z{\zeta}

\newcommand{\tsigma}{\tilde{\s}}
\newcommand{\tbsigma}{\tilde{\bsigma}}
\def\te{\tilde{\e}}
\def\tu{\tilde{u}}

\newcommand{\bchi}{\mbox{\boldmath$\chi$}}
\newcommand{\bdelta}{\mbox{\boldmath$\delta$}}
\newcommand{\bepsilon}{\mbox{\boldmath$\epsilon$}}
\newcommand{\bfeta}{\mbox{\boldmath$\eta$}}
\newcommand{\bgamma}{\mbox{\boldmath$\gamma$}}
\newcommand{\bomega}{\mbox{\boldmath$\omega$}}
\newcommand{\bvphi}{\mbox{\boldmath$\varphi$}}
\newcommand{\bphi}{\mbox{\boldmath$\phi$}}
\newcommand{\bPhi}{\mbox{\boldmath$\Phi$}}
\newcommand{\bpsi}{\mbox{\boldmath$\psi$}}
\newcommand{\bPsi}{\mbox{\boldmath$\Psi$}}
\newcommand{\bsigma}{\mbox{\boldmath$\sigma$}}
\newcommand{\btau}{\mbox{\boldmath$\tau$}}
\newcommand{\bxi}{\mbox{\boldmath$\xi$}}
\newcommand{\brho}{\mbox{\boldmath$\rho$}}
\newcommand{\bbeta}{\mbox{\boldmath$\beta$}}
\newcommand{\bzeta}{\mbox{\boldmath$\zeta$}}

\def\bk{\boldsymbol{\kappa}}
\def\bmu{\boldsymbol\mu}
\def\bxi{\boldsymbol{\xi}}
\def\bz{\boldsymbol{\zeta}}

\def\ba{{\bf a}}
\def\bb{{\bf b}}
\def\bc{{\bf c}}
\def\be{{\bf e}}
\def\bff{{\bf f}}
\def\bg{{\bf g}}
\def\bn{{\bf n}}
\def\bp{{\bf p}}
\def\bq{{\bf q}}
\def\bs{{\bf s}}
\def\bt{{\bf t}}
\def\bu{{\bf u}}
\def\bv{{\bf v}}
\def\bw{{\bf w}}
\def\bx{{\bf x}}
\def\by{{\bf y}}
\def\bzz{{\bf z}}

\def\bD{{\bf D}}
\def\bE{{\bf E}}
\def\bF{{\bf F}}
\def\bH{{\bf H}}
\def\bJ{{\bf J}}
\def\bV{{\bf V}}
\def\bU{{\bf U}}
\def\bW{{\bf W}}
\def\bX{{\bf X}}
\def\bY{{\bf Y}}

\def\cA{{\cal A}}
\def\cC{{\cal C}}
\def\cD{{\cal D}}
\def\cE{{\cal E}}
\def\cF{{\cal F}}
\def\cG{{\cal G}}
\def\cI{{\cal I}}
\def\cJ{{\cal J}}
\def\cK{{\cal K}}
\def\cL{{\cal L}}
\def\cO{{\cal O}}
\def\cP{{\cal P}}
\def\cQ{{\cal Q}}
\def\cR{{\cal R}}
\def\cS{{\cal \Sigma}}
\def\cT{{\cal T}}
\def\cU{{\cal U}}
\def\cV{{\cal V}}

\def\scT{{_\cT}}
\def\sD{{_D}}
\def\sE{{_E}}
\def\sF{{_F}}
\def\sFz{{_{F_z}}}
\def\sK{{_K}}
\def\sI{{_I}}
\def\sb{{_b}}
\def\sN{{_N}}

\def\curl{{{\bf curl} \ }}
\def\rot{{\mbox{rot}\ }}
\def\BPI{{\bf \Pi}}

\def\cth{\cT_h}
\def\ctH{\cT_H}

\def\tJ{\tilde{\J}}

\def\hK{\widehat{K}}
\def\hx{\widehat{x}}
\def\hy{\widehat{y}}
\def\bhv{\widehat{\bv}}

\def\l{\ell}
\def\bl{\boldsymbol{\ell}}
\def\col{\colon}
\def\f12{\frac12}
\def\dfrac{\displaystyle\frac}
\def\dint{\displaystyle\int}
\def\nab{\nabla}
\def\p{\partial}
\def\sm{\setminus}
\def\dsum{\displaystyle\sum}
\newcommand{\pp}[2]{\frac{\partial {#1}}{\partial {#2}}}
\def\bzero{{\bf 0}}

\def\divv{\nab\cdot}
\def\divx{\nab_x\cdot}
\def\divtx{\nab_{t,x}\cdot}
\def\nabx{\nab_x}

\newcommand{\grad}{\nabla}
\newcommand{\curlt}{{\nabla \times}}
\newcommand{\gperp}{\nabla^{\perp}}
\newcommand{\gradt}{\nabla\cdot}

\def\forallqq{\quad\forall\,}
\def\aph{A^{1/2}}
\def\amh{A^{-1/2}}

\def\osc{{\rm osc \, }}

\def\Im{{\rm Im}}
\newcommand{\tr}{{\rm tr}}
\def\divvr{{\rm div}}
\def\curllr{{\rm curl}}
\def\curll{{\rm curl}}
\def\curl{{\bf curl}}
\newcommand{\bgrad}{{\bf grad}}
\newcommand\diam{\mathrm{diam\,}}
\renewcommand\Im{\mathrm{Im\,}}
\def\Span{\mbox{Span}}
\def\supp{\mbox{supp\,}}
\newcommand{\trace}{{\rm trace}}

\newcommand{\tri}{|\!|\!|}
\newcommand{\ljump}{\lbrack\!\lbrack}
\newcommand{\rjump}{\rbrack\!\rbrack}
\newcommand{\bdm}{\begin{displaymath}}
\newcommand{\edm}{\end{displaymath}}
\newcommand{\beq}{\begin{equation}}
\newcommand{\eeq}{\end{equation}}
\newcommand{\beqa}{\begin{eqnarray}}
\newcommand{\eeqa}{\end{eqnarray}}
\newcommand{\beqas}{\begin{eqnarray*}}
\newcommand{\eeqas}{\end{eqnarray*}}
\newcommand{\ul}{\underline}
\newcommand{\wh}{\widehat}
\newcommand{\la}{\langle}
\newcommand{\ra}{\rangle}

\newcommand{\Lt}{L^2(\Omega)}
\newcommand{\Lts}{L^2(\Omega)^2}
\newcommand{\Ltc}{L^2(\Omega)^3}
\newcommand{\Ho}{H^1(\Omega)}
\newcommand{\Hoh}{H^1(\wh{\Omega})}
\newcommand{\Hoi}{H^1(\Omega_i)}
\newcommand{\Hos}{H^1(\Omega)^2}
\newcommand{\Hoc}{H^1(\Omega)^3}
\newcommand{\Hoch}{H^1(\wh{\Omega})^3}
\newcommand{\Hoci}{H^1(\Omega_i)^3}
\newcommand{\Hoz}{H^1_0(\Omega)}
\newcommand{\Ht}{H^2(\Omega)}
\newcommand{\Hti}{H^2(\Omega_i)}
\newcommand{\Hts}{H^2(\Omega)^2}
\newcommand{\Htc}{H^2(\Omega)^3}
\newcommand{\Htz}{H^0(\Omega)}
\newcommand{\Hh}{H^{1/2}(\Gamma)}
\newcommand{\Hhi}{H^{1/2}(\Gamma_i)}
\newcommand{\Hmh}{H^{-1/2}(\Gamma)}
\newcommand{\Hdiv}{H(\divvr;\,\Omega)}
\newcommand{\Hdivh}{H(\divv;\,\wh \Omega)}
\newcommand{\hcurl}{H(\curl\,A;\,\Omega)}
\newcommand{\Hcurl}{H(\curll\,A;\,\Omega)}
\newcommand{\Hcrl}{H(\curll\,;\,\Omega)}
\newcommand{\hcrl}{H(\curl\,;\,\Omega)}
\newcommand{\Hcrlh}{H(\curll\,;\,\wh\Omega)}
\newcommand{\hcrlh}{H(\curl\,;\,\wh\Omega)}
\newcommand{\Wdiv}{\BW_0(\mbox{\divv}\,;\,\Omega)}
\newcommand{\Wcurl}{\BW_0(\mbox{\curl}\,A;\,\Omega)}
\newcommand{\WcrossV}{\BW \times V}

\def\calS{{\cal S}}
\def\calT{{\cal T}}
\def\cB{{\cal B}}
\def\cH{{\cal H}}
\def\ba{{\mathbf{a}}}
\def\cM{{\cal M}}
\def\cN{{\mathcal{N}}}
\def\cE{{\mathcal{E}}}
\def\cT{{\mathcal{T}}}

\def\bE{{\bf E}}
\def\bS{{\bf S}}
\def\br{{\bf r}}
\def\bW{{\bf W}}
\def\bLambda{{\bf \Lambda}}

\newcommand{\lJump}{[\![}
\newcommand{\rJump}{]\!]}
\newcommand{\jump}[1]{[\![ #1]\!]}

\newcommand{\sd}{\bsigma^{\Delta}}
\newcommand{\st}{\tilde{\bsigma}}
\newcommand{\sh}{\hat{\bsigma}}
\newcommand{\rd}{\brho^{\Delta}}

\newcommand{\WH}{W\!H}
\newcommand{\NE}{N\!E}

\newcommand{\ND}{N\!D}
\newcommand{\BDM}{B\!D\!M}

\newcommand{\sT}{{_T}}
\newcommand{\sRT}{{_{RT}}}
\newcommand{\sBDM}{{_{BDM}}}
\newcommand{\sWH}{{_{WH}}}
\newcommand{\sND}{{_{ND}}}
\newcommand{\sV}{_\cV}

\newcommand{\dd}{\underline{{\mathbf d}}}
\newcommand{\C}{\rm I\kern-.5emC}
\newcommand{\R}{\rm I\kern-.19emR}
\newcommand{\W}{{\mathbf W}}
\def\3bar{{|\hspace{-.02in}|\hspace{-.02in}|}}
\newcommand{\A}{{\mathcal A}}

\title [BDM Mixed FEM in MATLAB]{An Efficient Implementation of 
Brezzi-Douglas-Marini (BDM) 
Mixed Finite Element Method in MATLAB}
\author[S. Zhang]{Shun Zhang}
\address{Department of Mathematics, City University of Hong Kong, Kowloon Tong, Hong Kong SAR, China}
\email{shun.zhang@cityu.edu.hk}
\urladdr{http://personal.cityu.edu.hk/~szhang26}
\thanks{This work was supported in part by
Research Grants Council of the Hong Kong SAR, China under the GRF Grant Project No. 11303914, CityU 9042090}

 \date{\today}

\keywords{MATLAB; mixed finite element method; Brezzi-Douglas-Marini element; Raviart-Thomas element; BDM element; RT element}

\maketitle
\begin{abstract}
In this paper, a MATLAB package \verb bdm_mfem  for a linear Brezzi-Douglas-Marini (BDM) mixed finite element 
method is provided for the numerical solution of elliptic diffusion problems with mixed boundary conditions on 
unstructured grids. BDM basis functions  defined by standard  barycentric coordinates are used in the paper. Local and global edge ordering are treated carefully. MATLAB build-in functions and vectorizations are used to guarantee the erectness of the programs. The package is simple and efficient, and can be easily adapted for more complicated edge-based finite element spaces.
A numerical example is provided to illustrate the usage of the package.

\end{abstract}

\section{Introduction}\label{intro}
\setcounter{equation}{0}

In recent years, MATLAB is widely used in the numerical simulation and is proved to be an excellent tool for 
academic educations. For example, Trefethen's book on spectral methods \cite{Tre:00} is extremely popular.
In the area of finite element method, there are several papers on writing clear, short, and easily adapted MATLAB 
codes, for example \cite{ACF:99,BC:05,Chen:09,FPW:11}.  Vectorizations are used in \cite{Chen:09} and \cite{FPW:11}
to guarantee the effectiveness of the MATLAB finite elements codes. The mixed finite element \cite{RaTh:77,BDM:85,BBF:13} is 
now widely used in many area of scientific computation. For example, in \cite{CaZh:09, CaZh:10a, CaZh:10b,CaZh:12,CaZh:15}, we use 
RT(Rviart-Thomas)/BDM(Brezzi-Douglas-Marini) space to build recovery-based a posteriori error estimators. On the other side, except for  the 
clear presentation of \cite{BC:05} on $RT_0$, the implementation of more complicated BDM elements is  still somehow confusing for 
researchers and students. The purpose of this paper is to fill this gap by giving a simple, efficient, and easily adaptable MATLAB 
implementation of $\BDM_1$-$P_0$ mixed finite element methods for of elliptic diffusion problems with mixed boundary conditions on 
unstructured grids.

For linear BDM elements,  there are several versions to write the basis functions explicitly. In \cite{BBF:13,SSD:03}, basis functions are defined on each 
elements using heights and normal vectors. This version of basis functions is less straightforward than the basis functions defined by 
barycentric coordinates. After all, everyone is very familiar with  barycentric coordinates in  finite element  programmings. Thus, in this 
implementation, we will use the definition which only uses barycentric coordinates. 

There are two basis functions on each edge for linear BDM elements. Unlike the RT element, BDM basis functions depend on the stating and terminal points of the edge. Thus, when assembling the local matrix, for a local edge on each element, we need to make sure we find its correct global stating and terminal points of the edge. There is an 
implementation of BDM element in $i$FEM package \cite{Chen:09} \texttt{https://bitbucket.org/ifem/ifem/}. But in order to make the local ordering of edges in an element is the same as the global ordering of edges, the triangles are not always counterclockwisely oriented. This will cause confusion for programmers. And if we use this ordering, sometime we may need two kinds of element map, one is counterclockwisely ordered, the other is ordered by the indices of vertices.  This will make things more complicated. In this implementation, we will use the standard counterclockwisely ordering of  vertices of triangles.

Besides these issues, to get the right convergence order, we need to handle the boundary conditions carefully. In this package, we use the basic data structure of $i$FEM \cite{Chen:09}, and full MATLAB vectorization is used.

In a summery, in this package:
\begin{enumerate}
\item $\BDM_1$ basis functions are explicitly defined using barycentric coordinates.
\item Elements are still positively oriented. A function is used to find the right global stating and terminal vertices of an edge in a local element. 
\item Mixed type of boundary conditions are handled correctly to guarantee the right order of convergence.
\item MATLAB build-in functions and vectorizations are used to guarantee the erectness of the programs. 
\end{enumerate}

The package can be download from \url{http://personal.cityu.edu.hk/~szhang26/bdm_mfem.zip}.

The paper is organized as follows. Section 2 describes the model diffusion
problem, its $\BDM_1$-$P_0$ mixed finite element approximation and the corresponding matrix problem.
Section 3 introduces a simple example problem to demonstrate out MATLAB code.
In Section 4, edge-based linear  BDM basis functions  are defined.
The main part of the code for the matrix problem is discussed in Section 6.
MATLAB functions to checking the errors are discussed in Section 7. 
Finally, we discuss some related finite elements in Section 8.

\section{Model problem and BDM1-P0 mixed finite element method}\label{modelproblem}
\setcounter{equation}{0}
\subsection{Model problem}
Let $\O$ be a bounded polygonal domain in $\Re^2$,
with boundary $\p \O = \bar{\Gamma}_\sD \cup\bar{ \Gamma}_\sN$,
$\Gamma_\sD\cap \Gamma_\sN = \emptyset$,  and $\mbox{measure}\,(\Gamma_D)\not= 0$,
and let $\bn$ be the outward unit vector normal to the boundary.
For a vector function $\btau = (\btau_1, \btau_2)$, define the divergence and curl operators by
$\gradt \btau = \frac{\p \btau_1}{\p x_1} + \frac{\p \btau_2}{\p x_2}$ and 
$\grad\times \btau = \frac{\p \btau_2}{\p x_1} - \frac{\p \btau_1}{\p x_2}$. For a function $v$, define the gradient and rotation operators by $\grad v = (\frac{\p v}{\p x_1}, \frac{\p v}{\p x_2})^t$ and  $\gperp v = (\frac{\p v}{\p x_2}, -\frac{\p v}{\p x_1})^t$. 

Consider diffusion equation
\begin{equation}\label{pde}
    -\gradt (\a(x) \grad u)  =  f   \quad\mbox{in} \quad  \O\\
\end{equation}
with boundary conditions
\beq\label{bc}
    -\a    \grad u \cdot \bn = g_N \quad\mbox{on} \quad \Gamma_N \quad
    \mbox{and}\quad u = g_D \quad\mbox{on}\quad \Gamma_D.
\eeq 
We assume that the right-hand side $f \in L^2(\O)$,  $g_D$ in $H^{1/2}(\Gamma_D)$,  
and that $g_N \in L^2(\Gamma_N)$, and that $\a(x)$ is a positive piecewise constant function.
Define the flux by $ \bsigma = -\a(x)\grad u$, then we have 
\beq \label{asg}
\a^{-1} \bsigma = -\grad u \quad \mbox{and} \quad \gradt \bsigma = f.
\eeq
Define the standard $\Hdiv$ spaces as 
\begin{eqnarray*}
\Hdiv &:=& \{\btau \in L^2(\O)^2 : \gradt \btau \in L^2(\O)\},\\
H_N(\divvr;\O)&:=& \{\btau \in \Hdiv : \btau \cdot \bn =0 \mbox{ on } \Gamma_N\}.
\end{eqnarray*}
Multiply the first equation in (\ref{asg}) by a $\btau \in  H_N(\divvr;\O)$ and
integrating by parts, we get
$$
(\a^{-1}\bsigma, \btau) = -(\grad u, \btau) = (\gradt\btau, v) - (\btau\cdot\bn , g_D)_{\Gamma_D}
$$
where 
$(\cdot,\cdot)_{\omega}$ is the $L^2$ inner product on a domain $\omega$. If $\omega=\O$, we omit the subscript.
Then the mixed variational formulation is to find $(\bsigma,\,u)\in
H(\divvr;\O)\times L^2(\O)$ with $\bsigma\cdot\bn = g_N$ on $\Gamma_N$, such that
\begin{equation}\label{mixed}
 \left\{\begin{array}{lclll}
 (\a^{-1}\bsigma,\,\btau)-(\divv \btau,\, u)&=&- (\btau\cdot\bn , g_D)_{\Gamma_D} \quad & \forall\,\, \btau \in
 H_N(\divvr;\O),\\[2mm]
 (\divv \bsigma, \,v) &=& (f,\,v)&\forall \,\, v\in L^2(\O).
\end{array}\right.
\end{equation}
The existence, uniqueness, and stability results of  (\ref{mixed}) are well-known, and can be found in standard references, for example, \cite{BBF:13}.

\subsection{$\BDM_1$-$P_0$ mixed formulation}
Let $\cT = \{K\}$ be a regular triangulation of the domain $\O$.
Denote the set of all edges of the triangulation by $
 \cE := \cE_I\cup\cE_D\cup\cE_N,
$ where $\cE_I$ is the set of all interior element edges and
$\cE_D$ and $\cE_N$ are the sets of all boundary edges
belonging to the respective $\Gamma_D$ and $\Gamma_N$. 
For any element $K\in\cT$, denote by $P_k(K)$ the space of polynomials on $K$ 
with total degree less than or equal to $k$. 
The $\Hdiv$ conforming  Brezzi-Douglas-Marini (BDM) space \cite{BDM:85} of the lowest order is defined by
$$
\BDM_1=\{\btau :\btau |_K \in \BDM_1(K),     \forall K \in \cT \} \quad \mbox{with}\quad \BDM_1(K)=P_1(K)^2.
$$
and let $\BDM_{1,N} = \BDM_1\cap H_N(\divvr;\O)$. The piecewise constant space $P_0$ is defined by 
$$
P_0 =\{v : v |_K \in P_0(K),  \;   \forall K \in \cT \}.
$$
For simplicity, we further assume that $\a$ is a positive constant in each element $K\in \cT$.
Let $\cT_N$ be the one dimensional mesh induced by $\cT$ on $\Gamma_N$.
Define
$$
P_1(\cT_N) = \{ v : v |_E \in P_1(E),   \;\; \forall E \in \cT_N \}.
$$
Let  $g_{N,h}$ be the $L^2$ projection of  $g_N$ on $P_1(\cT_N)$.
The $\BDM_1$-$P_0$ mixed finite element discrete problem then is to find $(\bsigma_h, u_h) \in
BDM_1 \times P_0$ with $\bsigma_h\cdot\bn =g_{N,h}$ on $\Gamma_N$, such that
\begin{equation}\label{problem_mixed}
 \left\{\begin{array}{lclll}
 (\a^{-1}\bsigma_h,\,\btau_h)-(\divv \btau_h,\, u_h)&=& - (\btau_h\cdot\bn , g_D)_{\Gamma_D}
 \quad & \forall\,\, \btau_h \in \BDM_{1,N},\\[2mm]
- (\divv \bsigma_h,\, v_h) &=& -(f,\,v_h)&\forall \,\, v_h\in P_0.
\end{array}\right.
\end{equation}
The discrete problem (\ref{problem_mixed}) has a unique solution, and the following error estimates hold assuming that the solution $(\bsigma, u)$ has enough regularity:
\beq \label{apriori}
\|\a^{-1/2} (\bsigma-\bsigma_h)\|_0 \leq Ch^2 \|\a^{-1/2}\bsigma\|_2 \quad \mbox{and}\quad
\|u-u_h\|_0 \leq C (h\|u\|_1 +  h^2 \|\bsigma\|_2 ),
\eeq
where $h$ is the mesh size and $\|\cdot\|_k$ is the standard norm of Sobolev space $H^k$. Thus, we should expect order 2 convergence of $\bsigma$ and order 1 convergence of $u$ if the solution is smooth enough. 

Let $\bsigma_N \in \BDM_1$ be the interpolation 
of $g_{N,h}$ on $\Gamma_N$ (The  construction of $\bsigma_N$ will be explained in Section \ref{BCrhside} in detail). Then $\bsigma_h = \bsigma_N + \bsigma_0$, with $\bsigma_0\in \BDM_{1,N}$ solves the following discrete problem:
\beq \label{discrevp}
\left\{\begin{array}{llll}
(\a^{-1} \bsigma_0, \btau_h) -(\divv \btau_h,\, u_h)&=& - (\btau_h\cdot\bn , g_D)_{\Gamma_D} - (\a^{-1} \bsigma_N, \btau_h) &
   \forall   \;\btau_h \in \BDM_{1,N},\\[2mm]
- (\divv \bsigma_0,\, v_h) &=& -(f,\,v_h)+(\divv \bsigma_N,\, v_h)& \forall \,\, v_h\in P_0.
\end{array}
\right.
\eeq

\subsection{Matrix problem}
Suppose that all edges are uniquely defined with fixed initial and terminal vertices, in Section \ref{section:BDMbasis}, we will define two basis 
functions $\bphi_{j,1}$ and $\bphi_{j,2}$ for an edge $E_j\in \cE$, $j = 1,\cdots, \NE$, with $\NE$ is the number of edges. 
For simplicity, we assume that the total number of all edges in $\cE_I$ and $\cE_D$ is $M$, and 
$\BDM_{1,N} = \mbox{span}\{ \bphi_{j,1}, \bphi_{j,2}, j=1\cdots,M\}$ (The real mesh may has a different order of indices).
The basis of $P_0$ is very simple. For the element $K_j$, its basis is $1_j$ which is $1$ on $K_j$ and $0$ elsewhere.

We want to compute the coefficient vectors $\bx \in \R^{2NE}$ of $\bsigma_h$ and $\by \in \R^{NT}$ of $u_h$ with respect to the $\BDM_1$ basis and $P_0$ basis
\beq
\bsigma_h = \sum_{j=1}^{\NE} \left( x_{j}\phi_{j,1}+x_{\NE+j}\phi_{j,2}\right) \quad \mbox{and}\quad
u_h = \sum_{k=1}^{NT} y_k 1_k,
\eeq
and
\beq \label{sigmaN}
\bsigma_N = \sum_{j=M+1}^{\NE} ( x_{j}\phi_{j,1}+x_{\NE+j}\phi_{j,2})
\eeq
 with  $x_j$ and $x_{\NE+j}$ are determined by  $g_{N,h}$ (discussed later in Section \ref{BCrhside}). 
Then
\begin{eqnarray*}
&&\sum_{j=1}^M \left( x_{j}(\a^{-1}\bphi_{i,1}, \bphi_{j,1}) + x_{j+\NE}(\a^{-1}\bphi_{i,1}, \bphi_{j,2}) \right)
-\sum_{k=1}^{NT} y_k (\bphi_{i,1},1)_{K_k} \\
&&\quad=- (g_D, \bphi_{i,1}\cdot\bn_\O)_{\Gamma_D} - \sum_{j=M+1}^{\NE} (x_j(\a^{-1}\bphi_{i,1}, \bphi_{j,1})+x_{j+\NE} (\a^{-1}\bphi_{i,1}, \bphi_{j,2})); \\
&&\sum_{j=1}^M \left( x_{j}(\a^{-1}\bphi_{i,2}, \bphi_{j,1}) + x_{j+\NE}(\a^{-1}\bphi_{i,2}, \bphi_{j,2}) \right)
-\sum_{k=1}^{NT} y_k (\bphi_{i,2},1)_{K_k} \\
&&\quad=- (g_D, \bphi_{i,2}\cdot\bn_\O)_{\Gamma_D} - \sum_{j=M+1}^{\NE} (x_j(\a^{-1}\bphi_{i,2}, \bphi_{j,1})+x_{j+\NE} (\a^{-1}\bphi_{i,2}, \bphi_{j,2}));\\
&& -\sum_{j=1}^{M} \left( x_j (\bphi_{j,1},1)_{K_\ell} +x_{j+\NE} (\bphi_{j,2},1)_{K_\ell}\right) \\
&& \quad= -(f, 1)_{K_\ell} + 
 \sum_{j=M+1}^{\NE} \left (x_j(\gradt \bphi_{j,1},1)_{K_{\ell}}+x_{j+\NE} (\gradt \bphi_{j,2}, 1)_{K_\ell} \right); 
\end{eqnarray*}
for $i=1,\cdots, M$, $\ell = 1,\cdots, NT$. The coefficients  $x_{j}$, $x_{j+\NE}$, $j=M+1,\cdots, \NE$ are known from $g_{N,h}$.
Rewrite it as a matrix problem, we have
\beq
\left(
\begin{array}{ccc}
  B & C^t     \\
  C&  0    
\end{array}
\right)
\left(
\begin{array}{ccc}
\bx    \\
\by    
\end{array}
\right) =
\left(
\begin{array}{ccc}
b1   \\
b2    
\end{array}
\right),
\eeq
where $B$ is a $2M\times 2M$ matrix and $C$ is a $NT\times 2M$ matrix with entries: 
\beq \begin{array}{lllllllll}
B_{i, j} = (\a^{-1}\bphi_{i,1}, \bphi_{j,1}), & B_{i, j+M} = (\a^{-1}\bphi_{i,1}, \bphi_{j,2}),  &
 B_{i+M, j} = (\a^{-1}\bphi_{i,2}, \bphi_{j,1}), 
 \\ B_{i+M, j+M} = (\a^{-1}\bphi_{i,2}, \bphi_{j,2}), & C_{\ell,j} = -(\gradt \bphi_{j,1}, 1_{\ell}), &
  C_{\ell, j+M} = -(\gradt \bphi_{ij2}, 1_{\ell}),
 \end{array}
 \eeq
 where  $i=1, \cdots, M$, $j=1,\cdots, M$, and $\ell = 1,\cdots, NT$;
 and the right hand side where $b1$ is a $2M\times 1$ vectors and  $b2$ is an $NT\times 1$ vector with entries
 \begin{eqnarray*}
b1_{i} &=& - (g_D, \bphi_{i,1}\cdot\bn_\O)_{\Gamma_D} - \sum_{j=M+1}^{\NE} (x_j(\a^{-1}\bphi_{i,1}, \bphi_{j,1})+x_{j+\NE} (\a^{-1}\bphi_{i,1}, \bphi_{j,2}));\\
b1_{i+M} &=& - (g_D, \bphi_{i,2}\cdot\bn_\O)_{\Gamma_D} - \sum_{j=M+1}^{\NE} (x_j(\a^{-1}\bphi_{i,2}, \bphi_{j,1})+x_{j+\NE} (\a^{-1}\bphi_{i,2}, \bphi_{j,2}));\\
b2_{\ell}  &=&  -(f, 1)_{K_\ell} + 
 \sum_{j=M+1}^{NE} (x_j(\gradt \bphi_{j,1},1)_{K_{\ell}}+x_{j+\NE} (\gradt \bphi_{j,2}, 1)_{K_\ell})
\end{eqnarray*}
where  $i=1, \cdots, M$ and $\ell = 1,\cdots, NT$;
 $(x_1, \cdots, x_M, x_{1+NE}, \cdots, x_{M+NE}, y_1,\cdots, y_{NT})^t$
 is the $2M+NT\times 1$ solution vector.

\section{A Numerical Example}\label{numericalexamples}
\setcounter{equation}{0}
We will demonstrate our MATLAB program by a simple test problem.
Let $\O = (-1,1)^2$, with $\Gamma_N = \{x\in (-1,1)\}\times\{y=1\}$, and the rest is $\Gamma_D$ The mesh is given in Fig. \ref{mesh}.
We choose the diffusion coefficient and the exact solution to be 
$$
\a = \left\{\begin{array}{lll}
10 & \mbox{if  } x<0; \\
1 & \mbox{if  } x>0; 
\end{array}
\right.
\quad \mbox{and}\quad
u(x,y) = \left\{\begin{array}{lll}
(x^2y^2+x)/10+y & \mbox{if  } x<0; \\
x^2y^2+x+y & \mbox{if  } x>0.
\end{array}
\right.
$$
The right-hand side $f=-2(x^2+y^2)$. 
The exact $\bsigma$ is 
$$
\bsigma (x,y) = \left\{\begin{array}{lll}
-(2xy^2+1, 2x^2y+10)^t & \mbox{if  } x<0; \\
-(2xy^2+1, 2x^2y+1)^t & \mbox{if  } x>0.
\end{array}
\right.
$$
It's clear that $\bsigma$ itself is not continuous, but its normal component is continuous.
The boundary conditions are
$$
u(x,y) = \left\{\begin{array}{lll}
 y^2/10+y-1/10& \mbox{if  } x=-1; \\
y^2+y+1 & \mbox{if  } x=1;\\ 
(x^2+x)/10-1 & \mbox{if  } x<0 \mbox{ and } y=-1; \\
x^2+x-1 & \mbox{if  } x>0\mbox{ and } y=-1;
\end{array}
\right.
$$
and
$$
-\a\grad u \cdot (0,1)^t = g_N= 
\left\{\begin{array}{lll}
-2x^2-10 & \mbox{if  } x<0 \mbox{ and } y=1; \\
-2x^2-1 & \mbox{if  } x>0 \mbox{ and } y=1.
\end{array}
\right.
$$
We choose this exact solution to emphasize that the flux $\bsigma$ is in $H(\divvr;\O)$, but not the $\grad u$.
The main MATLAB codes of $\a$, $f$, $g_D$ and $g_N$ are given in \mcode{exactalpha.m}, \mcode{f.m}, \mcode{gD.m}, and \mcode{gN.m} respectively. 
Here the Dirichlet boundary condition is given by the exact solution in \mcode{gD.m} for simplicity. 
\begin{lstlisting}[float,caption={exactalpha, f, gD, and gN}, label=exactalpha]
function z = exactalpha(p)
ix = (p(:,1)<0); z  = ones(size(p,1),1); z(ix) = 10.0;
end
function z = f(p) 
x = p(:,1); y = p(:,2); z = -2*(x.*x+y.*y);
end
function z = gD(p)
x = p(:,1); y = p(:,2); 
z=(x<0).*(0.1*x.*x.*y.*y+0.1*x+y)+(x>=0).*(x.*x.*y.*y+x+y);
end
function z = gN(p) 
x = p(:,1); y = p(:,2); z=(x<0).*(-2*x.*x-10)+(x>=0).*(-2*x.*x-1);
end
\end{lstlisting}

\begin{figure}[!hts]
\label{mesh}
    \begin{minipage}[!hbp]{0.48\linewidth}
        \centering
        \includegraphics[width=0.99\textwidth,angle=0]{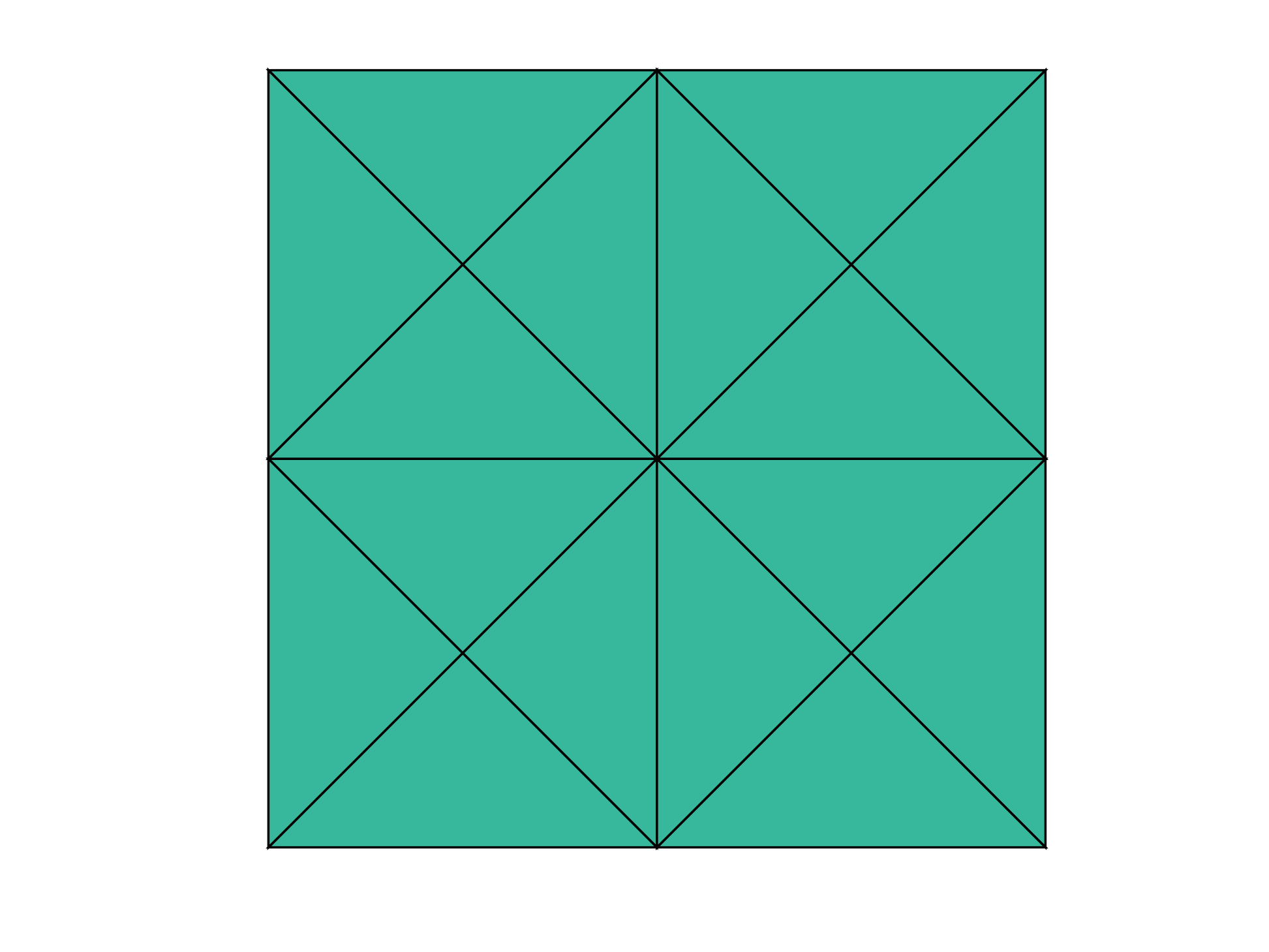}
        \caption{mesh of the example}%
        \end{minipage}%
        \label{mesh}
    \begin{minipage}[!hbp]{0.48\linewidth}
        \centering
        \includegraphics[width=0.99\textwidth,angle=0]{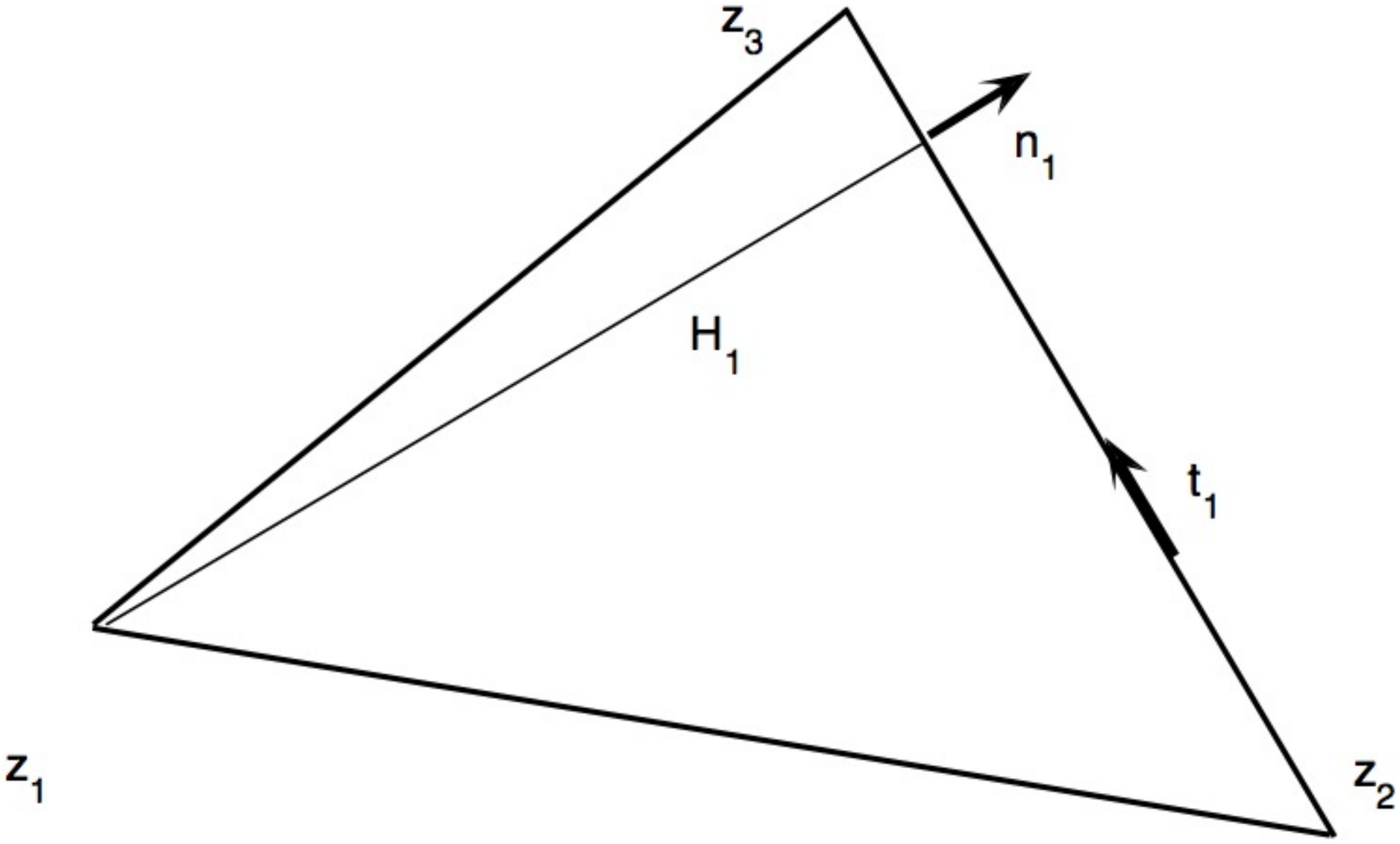}
        \caption{a triangle}%
        \end{minipage}%
        \label{triangle}
\end{figure}

The MATLAB program to solve this problem is given in \mcode{main.m}. 
Lines 2-7 read the essential geometric data of the problem. Line 9 generates the edges, and other necessary geometric relations.  Lines 11-12 solve the problem by the $\BDM_1$ mixed finite element method.

\begin{lstlisting}[float,caption={main.m}, label=main]
%% Geometry setting
node = [-1 1; 0 1; 1 1; -0.5 0.5; 0.5 0.5; -1 0; 0 0; 1 0; ...
    -0.5 -0.5; 0.5 -0.5;-1.0 -1.0; 0.0 -1.0;1.0 -1.0];
elem =[4 2 1;4 1 6;4 6 7;4 7 2;5 3 2;5 2 7;5 7 8;5 8 3;9 7 6;...
    9 6 11;9 11 12;9 12 7;10 8 7;10 7 12;10 12 13;10 13 8];
bdEdge = [2 0 0;1 0 0; 0 0 0;0 0 0;2 0 0;0 0 0;0 0 0;1 0 0;...
    0 0 0;1 0 0;1 0 0;0 0 0;0 0 0;0 0 0;1 0 0;1 0 0];
%% Geometric relations
[edge,elem2edge,signedge] = geomrelations(elem);
%%
[sigma,u] = diffusionbdm(node,elem,bdEdge,elem2edge,edge,...
     signedge,@exactalpha,@f,@gD,@gN);
\end{lstlisting}

\section{Triangulation and geometric data structures}\label{data}
\setcounter{equation}{0}
\subsection{Geometric description and geometric relations}
We follow \cite{Chen:09} for the data representation of the set of all vertices, the regular triangulation $\cT$, 
the edges, and the boundaries. 

The set of all vertices $\cN = \{z_1, \cdots, z_N\}$ is represented by an N$\times 2$ matrix   
\mcode{node(1:N,1:2)}, where $N$ is the number of vertices, and the $i$-th row of  \mcode{node} is the 
coordinates of the $i$-th vertex $z_i =(x_i, y_i)$, \mcode{node(i,:) = [xi,yi]}.  Lines 2-3 of \mcode{main.m} give  the \mcode{node} of our example. For example, the 1st vertex has coordinates $x_1=-1$ and $y_1 = 1$.

The triangulation $\cT$ is 
represented by an NT$\times 3$ matrix \mcode{elem(1:NT,1:3)} with NT the number of elements. 
The $i$-th element $K_i = \mbox{conv}\{z_i,z_j,z_k\}$
is stored as \mcode{elem(i,:) = [i j k]}, where the vertices are given in the counterclockwise order. Lines 4-5 of \mcode{main.m} give the \mcode{elem} of our example. For example, the 1st element $K_1$ has three vertices in the order of $z_4$, $z_2$ and $z_1$.

We call the the opposite edge of the $i$-th vertex, $i=1,2,3$ of a triangle the $i$-th edge of the triangle.

The matrix \mcode{bdEdge(1:NT, 1:3)} indicates which edge of an element is on the boundary of the domain.
 For a non-boundary edge, the value is $0$; the value is $1$ or $2$ for a Dirichlet or Neumann boundary edge respectively. Lines 4-5 of \mcode{main.m} give the \mcode{bdEdge} of our example. For example, the 1st element $K_1$ has three edge, the first edge is on the Neumann boundary, and the other two are not on the boundary.

\begin{lstlisting}[float,caption={geomrelations.m}, label=geomrelations]
function [edge,elem2edge,signedge] = geomrelations(elem)
NT = size(elem,1);
totalEdge = sort([elem(:,[2,3]); elem(:,[3,1]); elem(:,[1,2])],2);
[edge, useless, j] = unique(totalEdge, 'rows');
elem2edge = reshape(j,NT,3);

signedge = ones(NT,3);
signedge(:,1) = signedge(:,1) - 2* (elem(:,2)>elem(:,3));
signedge(:,2) = signedge(:,2) - 2* (elem(:,3)>elem(:,1));
signedge(:,3) = signedge(:,3) - 2* (elem(:,1)>elem(:,2));
end
\end{lstlisting}

The MATLAB code \mcode{geomrelations.m} generates \mcode{edge}, \mcode{elem2edge}, and \mcode{signedge}. The explanations of these codes can be found 
in \cite{Chen:09}.

The matrix \mcode{edge(1:NE,1:2)} is defined with NE the total number of edges, and the $k$-th edge $E_k$ with the starting vertex $z_i$ and the terminal 
vertex $z_j$ is stored as \mcode{edge(k,:) = [i j]}. We always ensure that \mcode{edge(k,1) < edge(k,2)}. Lines 3-4 of \mcode{geomrelations.m} 
generate  \mcode{edge}. For our example,  \mcode{NE=28} and \mcode{edge=[1     2; 1     4; 1     6;
     2     3;
     2     4;
     2     5;
     2     7;
     3     5;
     3     8;
     4     6;
     4     7;
     5     7;
     5     8;
     6     7;
     6     9;
     6    11;
     7     8;
     7     9;
     7    10;
     7    12;
     8    10;
     8    13;
     9    11;
     9    12;
    10    12;
    10    13;
    11    12;
    12    13]}. 
    For example, the 2nd edge's starting vertex is $z_1$ and terminal vertex is $z_4$.

For an edge $E_j$ with the starting vertex $z_s$ and the terminal vertex $z_t$, its normal direction is defined as 
$(y_t-y_s, x_s-x_t)^t/|E_j|$. 

The matrix \mcode{elem2edge(1:NT,1:3)}  is the matrix whose $k$-th row represents the 3 global indices of the edges in the order of local edges. 
Line 5 of \mcode{geomrelations.m} 
generates  \mcode{elem2edge}. For our example,   \mcode{elem2edge=
     [1     2     5;
     3    10     2;
    14    11    10;
     7     5    11;
     4     6     8;
     7    12     6;
    17    13    12;
     9     8    13;
    14    15    18;
    16    23    15;
    27    24    23;
    20    18    24;
    17    19    21;
    20    25    19;
    28    26    25;
    22    21    26].
}
For example, the 2nd element's three edges are $E_3$, $E_{10}$, and $E_2$.

Each edge has its fixed global orientation, while on each element, it has a local orientation, \mcode{localEdge=[2 3;3 1;1 2]}. 
The matrix \mcode{signedge} is an NT$\times 3$ matrix with value $1$ denoting the local and global orientations are the same, and $-1$ denoting that they are different. For our example,  \mcode{signedge=[-1     1    -1;
     1    -1    -1;
     1    -1     1;
    -1     1     1;
    -1     1    -1;
     1    -1    -1;
     1    -1     1;
    -1     1     1;
    -1     1    -1;
     1    -1    -1;
     1    -1     1;
    -1     1     1;
    -1     1    -1;
     1    -1    -1;
     1    -1     1;
    -1     1     1]}.

\subsection{Barycentric coordinate and its gradient}
 On an element $K$ with counterclockwise vertices $\{z_1,z_2,z_3\}$, we define the barycentric coordinate $\lambda_i =(a_i x+ b_i y +c_i)/(2|K|)$, $i=1,2,3$, such that 
$\lambda_i(z_j) = \delta_{ij}$. We only need to compute the gradient (and curl, rot, div) operators in this paper, so we only need to compute $a_i$, $b_i$, and the area $|K|$. The formulas are
$$
\grad \lambda_i = \dfrac{1}{2|K|} 
\left(
\begin{array}{ccc}
  y_{i+1} - y_{i+2} \\
  x_{i+2}-x_{i+1}
\end{array}
\right) \quad
\mbox{and}
\quad
2 |K|=
\det \left(
\begin{array}{ccc}
  x_2-x_1 & x_3-x_1 \\
  y_2-y_1& y_3-y_1
\end{array}
\right).
$$
The function \mcode{gradlambda.m} computes the coefficients \mcode{a}, \mcode{b}, and \mcode{area}. Here \mcode{a} and \mcode{b} are two NT$\times 3$ matrices, with each row stores the coefficients $a$ and $b$ for the three barycentric coordinates of corresponding 3 vertices.

\begin{lstlisting}[float,caption={gradlambda.m}, label=gradlambda]
function [a,b,area] = gradlambda(node,elem)
n1 = elem(:,1);	 n2 = elem(:,2);    n3 = elem(:,3);
NT = size(elem,1);   a = zeros(NT,3);   b = zeros(NT,3);
a(:,1) = node(n2,2)-node(n3,2);	 b(:,1) = node(n3,1)-node(n2,1);
a(:,2) = node(n3,2)-node(n1,2);	 b(:,2) = node(n1,1)-node(n3,1);
a(:,3) = node(n1,2)-node(n2,2);	 b(:,3) = node(n2,1)-node(n1,1);
area = (a(:,2).*b(:,3)-a(:,3).*b(:,2))/2.0;
end
\end{lstlisting}

\subsection{Normal and tangential vectors} \label{nortan}
On a local element $K$ with counterclockwise oriented vertices $\{z_{1}, z_2, z_3\}$, and $E_1 = \{z_2,z_3\}$, 
$E_2 = \{z_3,z_1\}$, and $E_3 = \{z_1,z_2\}$ (Figure \ref{triangle}), we will discuss the $\grad$ and $\gperp$ of $\lambda_2$  as examples:
$$
\gperp \lambda_2 =  \dfrac{1}{2|K|}
\left(
\begin{array}{ccc}
  y_3 - y_1   \\
  x_1- x_3 
\end{array}
\right)
\quad\mbox{and}
\quad
\gperp \lambda_2 =  \dfrac{1}{2|K|}
\left(
\begin{array}{ccc}
  x_1-x_3   \\
  y_1-y_3 
\end{array}
\right).
$$
where $|K|$ is the area of $K$.
On edge $E_1 =\{z_2,z_3\}$, the unit tangential and normal vectors are
$$
\bt_{E_1} = \dfrac{1}{|E_1|}
\left(
\begin{array}{ccc}
  x_3-x_2   \\
  y_3-y_2 
\end{array}
\right)
\quad\mbox{and}
\quad
\bn_{E_1} = \dfrac{1}{|E_{1}|}
\left(
\begin{array}{ccc}
  y_3-y_2   \\
  x_2-x_3
\end{array}
\right).
$$
where $|E|$ is the length of $E$.
From  Figure \ref{triangle},  we have
$\left(
\begin{array}{ccc}
  x_1-x_3   \\
  y_1-y_3 
\end{array}
\right) \cdot \bn_{E_1} =\left(
\begin{array}{ccc}
  y_3 - y_1   \\
  x_1- x_3 
\end{array}
\right) \cdot \bt_{E_1} = - H_1$, with $H_1$ is the height of $K_1$ on $E_1$. 
Now, by the fact  $2|K| = |E_1|\cdot H_1$, we have
$$
\gperp \lambda_2 \cdot\bn_{E_1} = \grad \lambda_2 \cdot\bt_{E_1}  =-1/|E_{1}|. 
$$
Similarly,
$$
\gperp \lambda_3 \cdot\bn_{E_1} = \grad \lambda_3 \cdot\bt_{E_1}  = 1/|E_{1}|. 
$$
On the hand, since the tangential and normal vectors of an edge are orthogonal, we have
\beq \label{orth}
\gperp \lambda_2 \cdot  \bn_{E_2} = \gperp \lambda_3 \cdot  \bn_{E_3} =0.
\eeq
Globally, on $E_{\ell} = \{z_s, z_t\}$ ($s<t$), the unit tangential and normal vectors are 
$$
\bt_{E_{\ell}} = \dfrac{1}{|E_{\ell}|}
\left(
\begin{array}{ccc}
  x_t-x_s   \\
  y_t-y_s 
\end{array}
\right)
\quad\mbox{and}
\quad
\bn_{E_{\ell}} = \dfrac{1}{|E_{\ell}|}
\left(
\begin{array}{ccc}
  y_t-y_s   \\
  x_s-x_t 
\end{array}
\right).
$$
We call the adjacent element whose out unit normal vector is the same as $\bn_{E_\ell}$ as $K^-_{\ell}$, and the element whose out unit normal vector is the opposite of $\bn_{E_\ell}$ as $K^+_{\ell}$.  On both $K_{\ell}^-$ and $K_{\ell}^+$, we have 
\beq \label{perpnormal}
\gperp \lambda_s \cdot\bn_{E_{\ell}}= \grad \lambda_s \cdot\bt_{E_{\ell}} = -1/|E_{\ell}| \quad\mbox{and}
\quad
\gperp \lambda_t \cdot\bn_{E_{\ell}} =
\grad \lambda_t \cdot\bt_{E_{\ell}}  =1/|E_{\ell}|.
\eeq

\section{Constructions of edge-based BDM basis functions} \label{section:BDMbasis}
\setcounter{equation}{0}

Let $\lambda_i$ be the standard linear Lagrange finite element basis function of the vertex $z_i$, i.e., it is piecewise linear on each element, globally continuous, $1$ at $z_i$ and $0$ at other vertices.

For an edge $E_{\ell} = \{z_s, z_t\}$, $s<t$, $1 \leq \ell\leq$ NE with the globally fixed  starting vertex $s$ and  terminal vertex $t$, its two $\BDM_1$ basis functions associated with the edge are
\beq \label{bdmbasis}
\bphi_{\ell,1} = \lambda_s \gperp \lambda_t \quad \mbox{and} \quad 
\bphi_{\ell,2} = -\lambda_t \gperp \lambda_s.
\eeq
It's clear that the basis functions are only non-zero in the two adjacent elements $K_{\ell}^-$ and $K_{\ell}^+$ (one element 
in the case of boundary) of $E_{\ell}$. 

\begin{lem}
There hold
\beq\label{propbdmbasis}
\bphi_{\ell,1} \cdot \bn_{E_{k}} |_{E_k} = \left\{ \begin{array}{ll}
0, & \mbox{if  } k \neq \ell;\\
\lambda_s/|E_\ell|, &  \mbox{if  }  k = \ell ;
\end{array}
\right.
\quad\mbox{and}
\quad
\bphi_{\ell,2} \cdot \bn_{E_{k}} |_{E_k}=  \left\{ \begin{array}{ll}
0, & \mbox{if  } k \neq \ell;\\
\lambda_t/|E_\ell|, &  \mbox{if  }  k = \ell,
\end{array}
\right.
\eeq
and
\beq \label{divBDM}
\gradt \bphi_{\ell,1} = \gradt \bphi_{\ell,2} = \dfrac{1}{2|K^-|} \mbox{ on } K^-\quad\mbox{and}\quad
\gradt \bphi_{\ell,1} = \gradt \bphi_{\ell,2} = -\dfrac{1}{2|K^+|} \mbox{ on } K^+.
\eeq
\end{lem}

\begin{proof}
By the discussion in Section \ref{nortan}, we have
$$
\gperp \lambda_s \cdot\bn_{E_{\ell}} = -1/|E_{\ell}| \quad\mbox{and}
\quad
\gperp \lambda_t \cdot\bn_{E_{\ell}} =1/|E_{\ell}|,
$$
and
$$
\bphi_{\ell,1} \cdot \bn_{E_{\ell}}|_{E_\ell} = \lambda_s /|E_{\ell}| \quad\mbox{and}
\quad
\bphi_{\ell,2} \cdot \bn_{E_{\ell}}|_{E_\ell} = \lambda_t/|E_{\ell}|.
$$
Assume $K_{\ell}^-$'s three counterclockwisely ordered vertices are $\{z_{r-}, z_s, z_t\}$.
We call edge with two endpoints $z_t$ and $z_{r-}$ to be $E_{t,r-}$, and the edge with two endpoint $z_{r-}$ and $z_s$ to be $E_{s,r-}$. Here  the orientations of these two edges are not important.
Then by (\ref{orth}), $\gperp \lambda_s\cdot \bn_{E_{t,r-}} =0$, and the fact $\lambda_t$ is zero on $E_{s,r-}$, we get 
$$
\bphi_{\ell,2}\cdot\bn_{E} = 0,  \mbox{  with  }  E = E_{s,r-} \mbox{  or  } E_{t,r-}.
$$
We can prove  (\ref{propbdmbasis}) is true for $\bphi_{\ell,1}$ and $K^+_{\ell}$ similarly.

The divergence part of the theorem can be easily proved by recalling the definition of $\lambda_s$ and $\lambda_t$ 
on $K^-_{\ell}$ and $K^{+}_{\ell}$, and notice that the counterclockwise ordering of vertices on $K^-_{\ell}$ is $\{z_{r-}, z_s, z_t\}$, and that on $K^+_{\ell}$ is $\{z_{r+}, z_t, z_s\}$.
\end{proof}

\begin{rem} \label{BDMRT}
There are other definitions of the $\BDM_1$ basis functions. 
One of the popular choice is the following:
\beq \label{hierarchical}
\bphi_{\ell,1} = \lambda_s \gperp \lambda_t - \lambda_t \gperp \lambda_s \quad \mbox{and} \quad 
\bphi_{\ell,2} =  \lambda_s \gperp \lambda_t+\lambda_t \gperp \lambda_s.
\eeq
The first basis function coincides with the $RT_0$ basis function with property $\bphi_{\ell,1}\cdot\bn_{E_{k}} = |E_{\ell}| \delta_{\ell k}$. This set of choices of basis functions is hierarchical. The reason we choose (\ref{bdmbasis}) is that it is symmetric. If the reader needs the basis to be  hierarchical, one should choose (\ref{hierarchical}). 
\end{rem}
\begin{rem}
The other versions of basis functions, for example, the RT basis function used in \cite{BC:05}, choose $\bphi^{rt}_{\ell} \cdot \bn_{E_k} = \delta_{k\ell}$. The issue of this choice of basis functions is that the mass matrix has a condition number 
$C h_{\max}/h_{\min}$, where $h_{\max}$ and $h_{\min}$ are the maximal and minimal diameters of the elements respectably. This will be a problem for  an adaptively generated meth. Though it can be easily fixed by preconditioning it with the inverse of the diagonal matrix of it, we avoid it by choosing basis in (\ref{bdmbasis}) or (\ref{hierarchical}).
\end{rem}

\section{Constructing and solving the matrix problem}\label{mfem}
\setcounter{equation}{0}

\subsection{Assembling matrices}

\subsubsection{Local matrices}
For an element $K$, we want to compute the local contributions of this element to the matrices $B$ and $C$.

We use  \mcode{localEdge=[2 3;3 1;1 2]} to denote the local edge with respect to the local indices of vertices. 
For the $i$-th edge, we let \mcode{ii1 = localEdge(i,1)} and    \mcode{ ii2 = localEdge(i,2)} to denote the local starting and terminal vertices of it.
When the local orientation and the global orientation of the edge $E$ is different (\mcode{signedge} of the edge is $-1$), 
we need to switch the local order to find the right global starting and terminal vertices of edge. 
For a given $i$-th local edge, by the line \mcode{i1 = (signedge(:,i)>0).*ii1+ (signedge(:,i)<0).*ii2}, we get \mcode{i1=ii1} if the local orientation and the  global orientation are the same, and \mcode{i1=ii2} otherwise. 
$i2$ can be done similarly.
Once we find the  $i1$ (local starting vertex) and $i2$ (local terminal vertex) with right global orientation, we can get the corresponding coefficients $a_{i1}$ and $b_{i1}$ of $\lambda_{i1} = (a_{i1}x+b_{i1}y+c_{i1})/(2|K|)$, and the same things for $i2$. 
The function \mcode{BDMrightorder.m} does the above job. With an input of $i = 1, 2$, or $3$ be the local vertex index of an element, this function returns the values \mcode{i1} and \mcode{i2}, which are the correct starting and terminal vertices of the corresponding edge with respect to the global fixed edge orientation, and \mcode{ai1,ai2,bi1,bi2} are the corresponding coefficients of $\lambda_{i1}$ and $\lambda_{i2}$.
 
\begin{lstlisting}[float,caption={BDMrightorder.m}, label=BDMrightorder]
function [i1,i2,ai1,bi1,ai2,bi2] = BDMrightorder(i,signedge,NT,a,b)
localEdge = [2 3; 3 1; 1 2];
ii1 = localEdge(i,1);               ii2 = localEdge(i,2);
i1 = (signedge(:,i)>0).*ii1+ (signedge(:,i)<0).*ii2;
i2 = (signedge(:,i)<0).*ii1+ (signedge(:,i)>0).*ii2;
ai1 = a((i1-1)*NT+(1:NT)');         ai2 = a((i2-1)*NT+(1:NT)');
bi1 = b((i1-1)*NT+(1:NT)');         bi2 = b((i2-1)*NT+(1:NT)');
end
\end{lstlisting}

To compute the local contribution of an element $K$ to $B$, the local mass matrix contains three cases, $(\a^{-1} \bphi_{i, 1}, \bphi_{j,1})$,  $(\a^{-1} \bphi_{i, 1}, \bphi_{j,2})$, and  $(\a^{-1} \bphi_{i, 2}, \bphi_{j,2})$ with 
$$
\bphi_{i, 1}= \lambda_{i1}\gperp \lambda_{i2}, \quad
\bphi_{i, 2}= -\lambda_{i2}\gperp \lambda_{i1},\quad
\bphi_{j, 1}= \lambda_{j1}\gperp \lambda_{j2},\quad \mbox{and}\quad
\bphi_{j, 2}= -\lambda_{j2}\gperp \lambda_{j2}.
$$
For barycentric coordinates, $\int_K \lambda_i \lambda_j dx = (1+\delta_{i,j})|K|/12$. 
An easy computation shows that
\begin{eqnarray*}
(\a^{-1} \bphi_{i, 1}, \bphi_{j,1})_K&=& \a^{-1}(1+\delta_{i1, j1})(a_{i2} \cdot a_{j2}+b_{i2}\cdot b_{j2})/(48|K|);\\
(\a^{-1} \bphi_{i, 1}, \bphi_{j,2})_K&=& -\a^{-1}(1+\delta_{i1, j2})(a_{i2} \cdot a_{j1}+b_{i2}\cdot b_{j1})/(48|K|);\\
(\a^{-1} \bphi_{i, 2}, \bphi_{j,2})_K&=& \a^{-1}(1+\delta_{i2, j2})(a_{i1} \cdot a_{j1}+b_{i1}\cdot b_{j1})/(48|K|).
\end{eqnarray*}
For  $k=1,2$, $\gradt \bphi_{i,k} = 1/(2|K|)$ when the $i$-th edge has the same orientation as the global edge, and  $-1/(2|K|)$ otherwise. Thus, denote $s(i)$ be the value of \mcode{signedge(:,i)}, we have
$$
-(\gradt \bphi_{i,1}, 1)_K =-(\gradt \bphi_{i,2}, 1)_K= - s(i)/2.
$$
The local matrix (here we abuse the notations to use local $\phi_{i,k}$, $i=1 \cdots 3$, $k=1,2$ to denote the local BDM basis functions): 
$$
-\left ((\grad \bphi_{1,1}, 1)_K, (\grad \bphi_{2,1}, 1)_K, (\grad \bphi_{3,1}, 1)_K, (\grad \bphi_{1,2}, 1)_K,(\grad \bphi_{2,2}, 1)_K, (\grad \bphi_{3,2}, 1)_K \right)
$$
 is simply $-(s(1), s(2), s(3),s(1),s(2), s(3))/2$, or \mcode{-[signedge(:,1:3), signedge(:,1:3)]} in MATLAB code.
\subsubsection{Assembling global matrices}
For a local index $i$, its corresponding global edge index is \mcode{double(elem2edge(:,i))}. MATLAB function \mcode{sparse} is used to generate the global matrices from local contributions. This is one of the key step to ensure the vectorization of the MATLAB finite element code,  see \cite{Chen:09, FPW:11} for more detailed discussions on  \mcode{sparse}. 

Here are some comments of the MATLAB code \mcode{assemblebdm.m}.
\begin{itemize}
\item Line 1: The function is called by 

 \mcode{A = assemblebdm(NT,NE,a,b,area,elem2edge,signedge,inva)}
 
where \mcode{A} is a \mcode{(2*NE+NT)*(2*NE+NT)} matrix, \mcode{a,b,area} are the coefficients of local barycentric coordinates, \mcode{signedge} is the \mcode{NT*3} matrix about the local and global edge orientations, and \mcode{inva} is \mcode{NT*1} vector of $\a^{-1}$.

\item Lines 4-15: We generates the matrix $B= (\a^{-1} \bsigma_h,\btau_h)$, $\bsigma_h$
 and $\btau_h$ in $\BDM_1$ by assembling local element-wise contributions. 

\item Lines 6-7: We generates the right starting and terminal indices of a global edge in the local element, and their corresponding coefficients of local barycentric coordinates.

\item Lines 8-10: We compute \mcode{E}, \mcode{H}, and \mcode{G}, which are  $(\a^{-1} \bphi_{i, 1}, \bphi_{j,1})_K$, 
$(\a^{-1} \bphi_{i, 1}, \bphi_{j,2})_K$, and $(\a^{-1} \bphi_{i, 2}, \bphi_{j,2})_K$, respectively.

\item Lines 11-13: $B$ is generated by \mcode{sparse}. 

\item Lines 16-20: $C$ is generated by local contributions \mcode{[-signedge(:,1:3), -signedge(:,1:3)]}.
\item Lines 21: \mcode{A} is generated by adding a zero matrix $D$ on $22$ block.
\end{itemize} 
\begin{lstlisting}[float,caption={assemblebdm.m}, label=assemblebdm]
function A = assemblebdm(NT,NE,a,b,area,elem2edge,signedge,inva)
B = sparse(2*NE, 2*NE); C = sparse(NT, 2*NE); D = sparse(NT,NT);
%% Blocal(6,6) = [E(3,3) H(3,3); H'(3,3) G(3,3)]
for i = 1:3
    for j = 1:3
        [i1,i2,ai1,bi1,ai2,bi2] = BDMrightorder(i,signedge,NT,a,b);
        [j1,j2,aj1,bj1,aj2,bj2] = BDMrightorder(j,signedge,NT,a,b);
        E =  inva./(48*area).*(1+(i1==j1)).*(ai2.*aj2+bi2.*bj2);
        H = -inva./(48*area).*(1+(i1==j2)).*(ai2.*aj1+bi2.*bj1);
        G =  inva./(48*area).*(1+(i2==j2)).*(ai1.*aj1+bi1.*bj1);
        ii = double(elem2edge(:,i));       jj = double(elem2edge(:,j));
        B = B + sparse(ii,jj,E,2*NE,2*NE)+sparse(ii,jj+NE,H,2*NE,2*NE)...
            +sparse(jj+NE,ii,H,2*NE,2*NE)+sparse(NE+ii,NE+jj,G,2*NE,2*NE);
    end
end
for i = 1:3
    ii = double(elem2edge(:,i));
    C = C + sparse(1:NT,ii,-signedge(:,i)/2,NT,2*NE)...
        +sparse(1:NT,ii+NE,-signedge(:,i)/2,NT,2*NE);
end
A = [B C';C D];
end
\end{lstlisting}


\subsection{Assembling the force $f$ term}
Since $\gradt \bsigma_h \in P_0$, we only need the numerical integration of the $f$ term to be accurate as if $f$ is a constant on each element. Thus, 
the term related to $-(f,1)_K$ can be computed by a one-point quadrature rule:
$$
-\int_K f dx \approx - f(x_{mid}, y_{mid}) |K|,
$$
where $(x_{mid}, y_{mid})$ are the coordinates of the gravity center of the element $K$. 

\subsection{Generating boundary data}
On the boundary, we need \mcode{sign_D} and \mcode{sign_N} to denote the difference between orientations inherited 
from the local edge ordering of the element and the global edge orientations like \mcode{edgesign}. Since the unit out 
 normal vector of an element on the boundary is the same as the unit out normal  vector of the whole domain,   
\mcode{sign_D} and \mcode{sign_N} 
are actually the difference of normal directions of Dirichlet and Neumann edges 
and global out normal directions.

The MATLAB function \mcode{boundary.m} generates the Dirichlet and Neumann edges and their orientations with respect to the global edges.

\begin{itemize}
\item Lines 3-10: We generates un-sorted Dirichlet and Neumann edges and their signs.
\item Lines 11-13: We generates sorted Dirichlet and Neumann edges and their signs.
\end{itemize}
\begin{lstlisting}[float,caption={boundary.m}, label=boundary]
function [Dirichlet, Neumann,sign_D,sign_N] = boundary(elem,bdEdge)
NT = size(elem,1);
totalEdge = [elem(:,[2,3]); elem(:,[3,1]); elem(:,[1,2])];
isBdEdge = reshape(bdEdge,3*NT,1);
Dirichlet = totalEdge((isBdEdge == 1),:);
Neumann = totalEdge((isBdEdge == 2),:);
NE_D = size(Dirichlet,1);               NE_N = size(Neumann,1);
sign_D = ones(NE_D,1);                  sign_N = ones(NE_N,1);
sign_D = sign_D(:,1) - 2* (Dirichlet(:,1)>Dirichlet(:,2));
sign_N = sign_N(:,1) - 2* (Neumann(:,1)>Neumann(:,2));
Dirichlet = sort(Dirichlet,2);          Neumann = sort(Neumann,2);
[Dirichlet,I_d] = sortrows(Dirichlet);  [Neumann,I_n] = sortrows(Neumann);
sign_D = sign_D(I_d);                   sign_N = sign_N(I_n);
end
\end{lstlisting}

We denote the set of indices of edges on the Dirichlet and Neumann boundary to be \mcode{ind_D} 
and \mcode{ind_N}, respectively. 

To computer the term $-(\btau\cdot\bn, g_D)_{\Gamma_D}$, we need to be careful about two things. One is the difference between unit out normal vector of the domain and that of the edge, the other one is the numerical quadrature formula.
In order to guarantee  the convergence order of $\BDM_1$ element, we should use two-point numerical quadrature on 
an edge. The 2-point Gauss-Legendre quadrature of a function $f(x)$ on interval $[a,b]$ is:
\beq\label{GL2}
\int_{a}^b f(x) dx \approx  \dfrac{b-a}{2}\left( f\left(\dfrac{a+b}{2} - \dfrac{b-a}{2\sqrt{3}} \right) 
+ f\left(\dfrac{a+b}{2} + \dfrac{b-a}{2\sqrt{3}} \right)
\right) .
\eeq
Note that on a Dirichlet edge $E_j =\{z_s,z_t\}$, $s<t$, $\bphi_{j,1} \cdot \bn_{E_j} = \lambda_s/|E_j|$ and
$\bphi_{j,2} \cdot \bn_{E_j} = \lambda_t/|E_j|$. Thus, to compute $-\int_{E_j} (\bphi_{j,\ell} \cdot\bn_{\O}) g_D dx$ by
 formula (\ref{GL2}), we need the value of $g_D$ at quadrature points 
$p_1 = \dfrac{n_1+n_2}{2} - \dfrac{n_2-n_1}{2\sqrt{3}}$ and $p_2 = \dfrac{n_1+n_2}{2} + \dfrac{n_2-n_1}{2\sqrt{3}}$, where $n_1$ and $n_2$ are the coordinates of the $z_s$ and $z_t$, respectively. We also need to know the value of 
$\lambda_s$ and $\lambda_t$ at $p_1$ and $p_2$, which are
$$
\lambda_s(p_1) = \lambda_t(p_2) = \dfrac{1}{2}+\dfrac{1}{2\sqrt{3}}\quad\mbox{and}\quad
\lambda_s(p_2) = \lambda_t(p_1) = \dfrac{1}{2}-\dfrac{1}{2\sqrt{3}}.
$$
Thus, by letting $s(j)=$ \mcode{sign_D(j)}, $s(j) = 1$ or $-1$ be the number denoting the difference between 
 the local and global edge orientation on $\Gamma_D$, we have 
\begin{eqnarray*}
-\int_{E_j} (\bphi_{j,1} \cdot\bn_{\O}) g_D dx & \approx & -s(j)\left( g_D(p_1)(\dfrac{1}{4}+\dfrac{1}{4\sqrt{3}})
+g_D(p_2)(\dfrac{1}{4}-\dfrac{1}{4\sqrt{3}})\right), \\
-\int_{E_j} (\bphi_{j,2} \cdot\bn_{\O}) g_D dx & \approx & -s(j) \left( g_D(p_1)(\dfrac{1}{4}-\dfrac{1}{4\sqrt{3}})
+ g_D(p_2)(\dfrac{1}{4}+\dfrac{1}{4\sqrt{3}})\right).
\end{eqnarray*}
To handle the Neumann boundary condition, on each $E_j = \{z_s,z_t \}\in \cE_N$, we want to compute
$\bsigma_N|_{E_j} = c_{j,1} \bphi_{j,1} + c_{j,2}\bphi_{j,2}$. Then 
$\bsigma_N\cdot\bn_{\O}|_{E_j} = c_{j,1} \bphi_{j,1}\cdot\bn_{\O} + c_{j,2}\bphi_{j,2}\cdot\bn_{\O}$. Let $s(j) = 1$ or $-1$ be the number denoting the difference between the local and global edge orientation on $E_j$, we should have 
$$ 
s(j)(c_{j,1} \lambda_s + c_{j,2}\lambda_t)/|E_j| \approx g_N.
$$
Let the $L^2$-projection of $g_N$ on to  $P_1(E_j) =\mbox{span}\{\lambda_s,\lambda_t\}$ to be
$g_{N,h}  = d_s \lambda_s + d_t \lambda_t$:
$$
\left(
\begin{array}{ccc}
(\lambda_s,\lambda_s)_{E_j}  &    (\lambda_t,\lambda_s)_{E_j} \\
(\lambda_s,\lambda_t)_{E_j}  & (\lambda_t,\lambda_t)_{E_j}     
\end{array}
\right)
\left(
\begin{array}{ccc}
d_s \\
d_t\end{array}
\right)
=
\left(
\begin{array}{ccc}
(g_N,\lambda_s)_{E_j} \\
(g_N,\lambda_t)_{E_j} 
\end{array}
\right).
$$
Since $\int_E \lambda_i \lambda_j ds =|E|(1+\delta_{ij}) /6$,  replace $(g_N,\lambda_s)_{E_j}$
and $(g_N,\lambda_t)_{E_j}$, by $I_{j,s}$ and $I_{j,t}$ using the two-point numerical quadrature (\ref{GL2}) as we did for $g_D$, we get
$$
d_s = (4I_{j,s}-I_2{j,t})/|E_j| 
\quad \mbox{and}\quad
d_t = (4I_{j,t}-2I_{j,s})/|E_j|. 
$$
So $\bsigma_N|_{E_j} = c_{j,1} \bphi_{j,1} + c_{j,2}\bphi_{j,2}$ with $c_{j,1}=s(j) (4I_{j,s}-2I_{j,t})$ and
$c_{j,2}=s(j)(4I_{j,t}-2I_{j,s})$, or, in the notation of (\ref{sigmaN}),
$$
x_j = s(j) (4I_{j,s}-2I_{j,t}) \quad \mbox{and} \quad x_{\NE+j} = s(j)(4I_{j,t}-2I_{j,s}).
$$
With this know $\bsigma_N$, then the right-hand side of the discrete problem can be easily handled
as in \cite{ACF:99,Chen:09}.

\begin{lstlisting}[float,caption={rhside.m}, label=rhside]
function [A,b,sol,freeDof]=rhside(node,elem,edge,bdEdge,area,A,sol,f,gD,gN)
NT = size(elem,1);  NE = size(edge,1); 
[Dirichlet, Neumann,sign_D,sign_N] = boundary(elem,bdEdge);
%% Assemble right hand side.
mid = (node(elem(:,1),:)+node(elem(:,2),:)+node(elem(:,3),:))/3;
b2 = accumarray((1:NT)',-f(mid).*area,[NT 1]);
%% Drichelet BC
n1= node(Dirichlet(:,1),:);            n2= node(Dirichlet(:,2),:); 
p1=(n1-n2)/2*sqrt(1/3)+(n2+n1)/2;      p2=(n2-n1)/2*sqrt(1/3)+(n2+n1)/2;
intgDphi1n = (gD(p1)*(1+sqrt(1/3))+gD(p2)*(1-sqrt(1/3)))/4;
intgDphi2n = (gD(p2)*(1+sqrt(1/3))+gD(p1)*(1-sqrt(1/3)))/4;
[useless, ind_D] = intersect(edge, Dirichlet, 'rows');
bb1 = accumarray(ind_D,-sign_D.*intgDphi1n,[NE 1]);
bb2 = accumarray(ind_D,-sign_D.*intgDphi2n,[NE 1]);
b1 = [bb1;bb2];
%% Neumann BC
n1= node(Neumann(:,1),:);              n2= node(Neumann(:,2),:); 
p1=(n1-n2)/2*sqrt(1/3)+(n2+n1)/2;      p2=(n2-n1)/2*sqrt(1/3)+(n2+n1)/2;
edgeLength_N = sqrt(sum((n1-n2).^2,2));
intgNlams =  edgeLength_N.*(gN(p1)*(1+sqrt(1/3))+gN(p2)*(1-sqrt(1/3)))/4;
intgNlamt =  edgeLength_N.*(gN(p2)*(1+sqrt(1/3))+gN(p1)*(1-sqrt(1/3)))/4;
[useless, ind_N] = intersect(edge, Neumann, 'rows');
sol(ind_N) = 2*sign_N.*(2*intgNlams-intgNlamt); 
sol(ind_N+NE) = 2*sign_N.*(2*intgNlamt-intgNlams);
%% modify right hand side
b = [b1;b2];    b = b - A*sol;
%%freeDof
isBdDof = false(2*NE+NT,1); isBdDof(ind_N)=true; isBdDof(ind_N+NE)=true;
freeDof = find(~isBdDof);
end
\end{lstlisting}

The followings are some comments about the function \mcode{rhside.m} of handling the $f$ term and boundary conditions.
\begin{itemize}
\item Lines 5-6: We generate terms related to $-(f,v)$.
\item Lines 8-15: We generate terms related to $-(\btau\cdot\bn, g_D)_{\Gamma_D}$.
\item Lines 17-24: We construct the $\bsigma_N$.
\item Line 26: We handle the right-hand side of the matrix problem.
\item Lines 28-29: We find the degrees of freedom of the matrix problem. (Those terms related to $\bsigma_N$ are known, thus not free.)
\end{itemize}

\subsection{Solving the BDM mixed problem}

The MATLAB function \mcode{diffusiondbm.m} is the principle function to solve the BDM mixed problem. With all the building blocks given before, it is relatively easy. Line 5 is used to got $\a^{-1}$ in each element. All the rest lines are self-explanatory.  

\begin{lstlisting}[float,caption={diffusionbdm.m}, label=diffusionbdm]
function [sigma,u] = diffusionbdm(node,elem,bdEdge,elem2edge,edge,...
    signedge,exactalpha,f,gD,gN)
NT = size(elem,1);  NE = size(edge,1); 
sol = zeros(2*NE+NT,1);
inva =1./exactalpha((node(elem(:,1))+node(elem(:,2))+node(elem(:,3)))/3); 
[a,b,area] = gradlambda(node,elem);
A = assemblebdm(NT,NE,a,b,area,elem2edge,signedge,inva);
[A,b,sol,freeDof] = rhside(node,elem,edge,bdEdge,area,A,sol,f,gD,gN);
sol(freeDof) = A(freeDof,freeDof)\b(freeDof);
sigma = sol(1:2*NE);    u = sol(2*NE+1:end);
end
\end{lstlisting}

\section{Checking errors}
The final step of a numerical test is often the convergence test by computing some norms of the error between the 
exact and numerical solutions obtained for different mesh sizes. In our case, we need to compute 
$$
\|\a^{-1/2}(\bsigma- \bsigma_h)\|_0 \quad \mbox{and} \quad 
\|u- u_h\|_0.
$$

For the both  terms $\|\a^{-1/2}(\bsigma- \bsigma_h)\|_0$ and $\|u- u_h\|_0$, the direct way to compute them is summing up the errors on each element by direct computations. 
To lower the numerical quadrature order, the first step should be 
\begin{eqnarray*}
\|\a^{-1/2}(\bsigma-\bsigma_h)\|_0^{2} & =  &
(\a^{-1}(\bsigma-\bsigma_h),\bsigma-\bsigma_h) = (\a^{-1}\bsigma,\bsigma)-2(\a^{-1}\bsigma,\bsigma_h)+(\a^{-1}\bsigma_h,\bsigma_h),\\
\|u-u_h\|_0^{2} & =  &
(u-u_h,u-u_h) = (u,u)-2 (u,u_h)+ (u_h,u_h).
\end{eqnarray*}

The terms $(\a^{-1}\bsigma,\bsigma)$ and $(u,u)$ can be computed exactly by softwares like Mathematica and Maple. In our example, $(\a^{-1}\bsigma,\bsigma) =1993/75 \approx 26.5733$ and $(u,u) = 18131/7500 \approx 2.41747$. The terms $(\a^{-1}\bsigma_h,\bsigma_h)$ and $(u_h,u_h)$ can also be computed exactly. 
For the terms $(\a^{-1}\bsigma,\bsigma_h)$ and $(u_h,u_h)$, numerical quadratures must be used. 

Since this part of codes is less important, we only give brief comments about the functions we used. 
The MATLAB function \mcode{exactsigma.m} returns the exact value of $\bsigma$.
The  function \mcode{sigmahBDM.m}  returns the values of $\bsigma_h$ at those numerical quadrature points.
The  function \mcode{errorBDM.m} is used to compute the errors. Here, we use a 6-point quadrature to compute them. 
Matrix \mcode{xw} is a $6*3$ matrix, with \mcode{xw(:,1:2)} are the coordinates such that the quadrature point is 
\mcode{p=p1*(1-xw(i,1)-xw(i,2))+p2*xw(i,1)+p3*xw(i,2)}, where $p1$, $p2$, and $p3$ are the coordinates of three vertices;  \mcode{xw(:,3)} is the corresponding weight for the point.

\begin{lstlisting}[float,caption={exactsigma.m}, label=exactsigma]
function [sigma1,sigma2] = exactsigma(p)
x = p(:,1); y = p(:,2);
sigma1=-2*x.*y.*y-1;
sigma2=(x<0).*(-2*x.*x.*y-10)+(x>=0).*(-2*x.*x.*y-1);
end
\end{lstlisting}

\begin{lstlisting}[float,caption={sigmahBDM.m}, label=sigmahBDM]
function [sigmah1,sigmah2] = ...
    sigmahBDM(elem,node,NE,elem2edge,signedge,w1,w2,sigma)
[a,b,area] = gradlambda(node,elem); NT = size(elem,1);
lambda_at_w = [1-w1-w2,w1,w2];
sigmah1 = zeros(NT,1);          sigmah2 = zeros(NT,1);
for i = 1:3
    [i1,i2,ai1,bi1,ai2,bi2] = BDMrightorder(i,signedge,NT,a,b);
    ii = double(elem2edge(:,i));
    sigmah1 = sigmah1+sigma(ii).*lambda_at_w(i1)'.*bi2./area/2 ...
        -sigma(ii+NE).*lambda_at_w(i2)'.*bi1./area/2;
    sigmah2 = sigmah2-sigma(ii).*lambda_at_w(i1)'.*ai2./area/2 ...
        +sigma(ii+NE).*lambda_at_w(i2)'.*ai1./area/2;    
end
end
\end{lstlisting}

\begin{lstlisting}[float,caption={errorBDM.m}, label=errorBDM]
function [error_sigma, error_u] = errorBDM(node,elem,NE,area,elem2edge,...
    signedge,inva,sigma,u,ss,uu)
n1 = elem(:,1);  n2 = elem(:,2);  n3 = elem(:,3);
p1 = node(n1,:); p2 = node(n2,:); p3 = node(n3,:);
xw=[0.44594849091597 0.44594849091597 0.22338158967801;...
    0.44594849091597 0.10810301816807 0.22338158967801;...
    0.10810301816807 0.44594849091597 0.22338158967801;...
    0.09157621350977 0.09157621350977 0.10995174365532;...
    0.09157621350977 0.81684757298046 0.10995174365532;...
    0.81684757298046 0.09157621350977 0.10995174365532];
NP =size(xw,1); NT = size(elem,1);
uuhlocal = zeros(NT,1); % (u, u_h)_K
sshlocal = zeros(NT,1); % (\a^^{-1}sigma, sigma_h)_K
shshlocal = zeros(NT,1);% (\a^^{-1}sigma_h, sigma_h)_K
for i=1:NP
    p=p1*(1-xw(i,1)-xw(i,2))+p2*xw(i,1)+p3*xw(i,2);
    uuhlocal = uuhlocal + exactu(p)*xw(i,3);
    [sigma1,sigma2] = exactsigma(p);
    [sigmah1,sigmah2] = ...
        sigmahBDM(elem,node,NE,elem2edge,signedge,xw(i,1),xw(i,2),sigma);
    sshlocal = sshlocal + (sigma1.*sigmah1+sigma2.*sigmah2)*xw(i,3);
    shshlocal = shshlocal + (sigmah1.*sigmah1+sigmah2.*sigmah2)*xw(i,3);
end
uuhlocal = u.*uuhlocal.*area;           uhuhlocal = u.*u.*area;
sshlocal = inva.*sshlocal.*area;        shshlocal = inva.*shshlocal.*area;
error_u = sqrt(abs(uu-2*sum(uuhlocal)+sum(uhuhlocal)));
error_sigma = sqrt(abs(ss-2*sum(sshlocal)+sum(shshlocal)));
end\end{lstlisting}

We use the code listed in Listing \ref{checkingerrors} to compute the error.

\begin{lstlisting}[float, caption={checking errors}, label=checkingerrors]
uu=18131/7500; ss=1993/75;
[error_sigma, error_u] = errorBDM(node,elem,NE,area,elem2edge,...
    signedge,inva,sigma,u,ss,uu);
\end{lstlisting}

The original mesh  given has mesh size $h=1$. We refine the mesh uniformly several times, for example, using  \mcode{[node,elem,bdEdge]=uniformbisect(node,elem,bdEdge)}, where \mcode{uniformbisect} is a uniform refinement MATLAB function given in the package $i$FEM \cite{Chen:09}. We compute the errors for different mesh sizes, and have the Table \ref{error}. We get
$$
\dfrac{
\|\a^{-1}(\bsigma-\bsigma_h)\|_0}
{\|\a^{-1}(\bsigma-\bsigma_{h/2})\|_0} \approx 4 \quad\mbox{and}\quad
\dfrac{
\|u-u_h\|_0}
{\|u-u_{h/2}\|_0} \approx 2.
$$ 
This result is  in agreement with the a priori error esteems (\ref{apriori}).

\begin{table}[htdp]
\caption{Errors of $\bsigma$ and $u$ on different mesh sizes}
\begin{center}
\begin{tabular}{|c |c | c | c| c | }
\hline
$h$ & $\|\a^{-1}(\bsigma-\bsigma_h)\|_0$ &  $\dfrac{\|\a^{-1}(\bsigma-\bsigma_h)\|_0}
{\|\a^{-1}(\bsigma-\bsigma_{h/2})\|_0} $ & $\|u-u_h\|_0$ & $\dfrac{
\|u-u_h\|_0}
{\|u-u_{h/2}\|_0}$ \\
\hline
1 &     1.6968e-01  &   &       4.9712e-01  &  \\
\hline
1/2 &        4.2091e-02  & 4.0314 &     2.4400e-01 &   2.0374 \\
\hline
1/4 & 1.0600e-02 &  3.9707 &  1.2118e-01 & 2.0135   \\
\hline
1/8 &    2.6630e-03 & 3.9805&    6.0481e-02 &  2.0037 \\
\hline
1/16 & 6.6739e-04 &  3.9901 &    3.0226e-03 &   2.0009 \\
\hline
1/32 & 1.6705e-04 & 3.9952 & 1.5111e-03 & 2.0002\\
\hline
1/64 & 4.1788e-05 & 3.9975 &    7.5555e-04 &   2.0001\\
\hline
\end{tabular}
\end{center}
\label{error}
\end{table}%

\section{Related finite elements}\label{RT_Nedlec}
\setcounter{equation}{0}

Once we know how to programming $\BDM_1$ methods, we can actually write codes for similar finite elements, with two things in mind: Writing basis functions in barycentric coordinates and correcting the local edge orientation.  For more higher order elements, we may use 
Legendre polynomials  instead of Lagrange polynomials, but the idea is similar.
\subsection{Other H(div)-conforming edge-based basis in 2D}
%

For $RT_0$ element, as mentioned in Remark \ref{BDMRT}, its basis function on an edge can be written as
$$
\bphi^{rt}_{\ell} = \lambda_s \gperp \lambda_t - \lambda_t \gperp \lambda_s
$$
Actually, it is insensitive to the staring and terminal vertices: 
$$
\bphi^{rt}_{\ell} = \lambda_s \gperp \lambda_t - \lambda_t \gperp \lambda_s = 
\lambda_t \gperp \lambda_s - \lambda_s \gperp \lambda_t,
$$
and on adjacent elements $K^- =\{z_{r-}, z_s, z_t\}$ and $K^+=\{z_{r+}, z_t, z_s\}$,
$$
\bphi^{rt}_{\ell}|_{K^-} = \dfrac{1}{2|K^-|} 
\left(
\begin{array}{ccc}
x-x_{r-} \\
y-y_{r-}
\end{array}
\right)
\quad\mbox{and}\quad
\bphi^{rt}_{\ell}|_{K^+} =- \dfrac{1}{2|K^+|} 
\left(
\begin{array}{ccc}
x-x_{r+} \\
y-y_{r+}
\end{array}
\right).
$$
Thus lines 4-5 of \mcode{BDMrightorder} is unnecessary for $RT_0$ elements.

For all $RT_k$ and $\BDM_k$ spaces, basis functions are divided into two categories, edge based functions and 
element-based functions. The element based functions are usually easy to construct since they are only non-zero in 
one element, see \cite{BBF:13}. So we will only discuss edge based basis functions. Both $RT_k$ and $\BDM_k$ spaces have $k+1$ basis functions on each edge. We only need to functions to span $P_k(E_{\ell})$ on $E_{\ell}$.
For $\BDM_2$ and $RT_2$, three edge basis functions on $E_{\ell}$ are
$$
\bphi^{1}_{\ell,1} = \lambda_s^2  \gperp \lambda_t, \quad
\bphi^{1}_{\ell,2} = \lambda_s \lambda_t (\gperp \lambda_t - \gperp \lambda_s)  , \quad\mbox{and}\quad
\bphi^{1}_{\ell,3} = -\lambda_t^2 \gperp \lambda_s.
$$
That is, we use  $\lambda_s^2$, $\lambda_t^2$, and $\lambda_s\lambda_t$ to span $P_2$ on $E_{\ell}$. We can use other choices, for example, Legendre polynomials for high order spaces. With this observation, the code developed in this paper can be easily adapted  to these spaces.

\subsection{N\'ed\'elec spaces in 2D}
For N\'ed\'elec spaces, the 1st and 2nd types  N\'ed\'elec  spaces correspond to their $H(\divvr)$ counterparts are $RT$ and $BDM$ spaces, respectively. We also only need to discuss the construction of basis functions on edges.
By (\ref{perpnormal}), the zero-moment   $H(\mbox{curl})$ edge basis function is
$$
\bpsi^{ned}_{\ell} = \lambda_s \grad \lambda_t - \lambda_t \grad \lambda_s.
$$
The linear $H(\mbox{curl})$ edge basis functions are
$$
\bpsi_{\ell,1}^{ned} = \lambda_s \grad \lambda_t \quad \mbox{and}\quad
\bpsi_{\ell,2}^{ned} = - \lambda_t \grad \lambda_s.
$$
\subsection{$H(\divvr)$ and $H(\mbox{curl})$ and  basis functions in 3D} For a tetrahedral mesh, $H(\divvr)$ basis functions are defined on faces. Like the
\mcode{edge} structure in this paper, we should define a \mcode{face} matrix. If $F_{\ell} = \{z_r, z_s, z_t\}$, we need ensure $r<s<t$, and 
choose the normal direction such that the area of $\{z_r, z_s, z_t\}$ is positive. Once this is done, we can do similar things as in 2D.
The $BDM_1$ basis functions in 3D on a face $F_{\ell} = \{z_r, z_s, z_t\}$ are 
$$
\lambda_r \grad \lambda_s \times \grad \lambda_t, \quad \lambda_s \grad \lambda_t \times \grad \lambda_r, \quad
\mbox{and} \quad \lambda_t \grad \lambda_r \times \grad \lambda_s.
$$
Its $RT_0$ basis function on $F$ is 
$$
\lambda_r \grad \lambda_s \times \grad \lambda_t+\lambda_s \grad \lambda_t \times \grad \lambda_r+ \lambda_t \grad \lambda_r \times \grad \lambda_s.
$$
For three dimensionalal $H(\mbox{curl})$ space, basis functions in barycentric coordinates can be found in \cite{GDG:05}.
%

\end{document}